\documentclass[12pt, reqno]{amsart}

\usepackage{amsmath}
\usepackage{amsfonts}
\usepackage{dsfont}
\usepackage{a4wide}
\usepackage{amssymb}
\usepackage{amsthm}
\usepackage[colorlinks, citecolor=blue, linkcolor=red]{hyperref}

\theoremstyle{plain}

\newtheorem{teo}{Theorem}[section]
\newtheorem{lemma}[teo]{Lemma}
\newtheorem{prop}[teo]{Proposition}
\newtheorem{cor}[teo]{Corollary}
\newtheorem{ackn}{Acknowledgments\!}

\theoremstyle{definition}

\newtheorem{conge}[teo]{Conjecture}

\theoremstyle{remark}

\newtheorem{rem}[teo]{Remark}

\def\ricc{{\mathrm {Ric}}}

\numberwithin{equation}{section}

\def\RR{{\mathbb R}}

\def\gt{\widetilde{g}}

\def\gt{\widetilde{g}}

\def\eps{\varepsilon}

\newcommand{\set}[1]{{\left\{#1\right\}}}               
\newcommand{\abs}[1]{{\left|#1\right|}}                 

\title[Uniqueness of critical metrics for a quadratic curvature functional]{Uniqueness of critical metrics for a quadratic curvature functional}
\date{\today}

\author[Giovanni Catino]{Giovanni Catino}
\address[Giovanni Catino]{Dipartimento di Matematica, Politecnico di Milano, Piazza Leonardo da Vinci 32, 20133 Milano, Italy}
\email[]{giovanni.catino@polimi.it}

\author[P. Mastrolia]{Paolo Mastrolia}
\address[Paolo Mastrolia]{Dipartimento di Matematica, Universit\`{a} degli Studi di Milano, Via Saldini 50, 20133 Italy.}
\email[]{paolo.mastrolia@unimi.it}

\author[D. D. Monticelli]{Dario D. Monticelli}
\address[Dario Monticelli]{Dipartimento di Matematica, Politecnico di Milano, Piazza Leonardo da Vinci 32, 20133 Milano, Italy}
\email[]{dario.monticelli@polimi.it}

\begin{document}

\begin{abstract}
In this paper we prove a new rigidity results for complete, possibly non-compact, critical metrics of the quadratic curvature functional $\mathfrak{S}^2 = \int R_g^{2} dV_g$: we show that critical metrics $(M^n, g)$ with finite energy are always scalar flat, i.e.  global minima, provided $n\geq 10$.
\end{abstract}

\maketitle

\begin{center}


\medskip

\centerline{\bf AMS subject classification: 53C21, 53C24, 53C25}

\end{center}

\


\section{Introduction}

This paper is devoted to the study of critical metrics for the quadratic curvature functional
$$
\mathfrak{S}^2 = \int R_g^{2} dV_g.
$$
To fix the notation, let $M^{n}$, $n\geq 2$, be a $n$--dimensional smooth  manifold without boundary. Given a Riemannian metric $g$ on $M^n$, we denote with $\operatorname{Riem_g}$, $W_g$, $\ricc_g$ and $R_g$, respectively, the Riemann curvature tensor, the Weyl tensor, the Ricci tensor and the scalar curvature. It is well known that a basis for the space of quadratic curvature functionals, defined on the space of smooth metrics on $M^n$, is given by
$$
\mathfrak{W}^2 = \int |W_g|^{2} dV_g\,, \qquad \mathfrak{r}^2 = \int |\ricc_g|^{2} dV_g\,, \qquad\, \mathfrak{S}^2 = \int R_g^{2} dV_g.
$$
The only quadratic functional in the case $n=2$ is given by $\mathfrak{S}^2$, while in dimension $n=3$ one only has $\mathfrak{S}^2$ and $\mathfrak{r}^2$.
From the standard decomposition of the Riemann tensor, for every $n\geq4$, one has
$$
\mathfrak{R}^2 = \int |\operatorname{Riem_g}|^{2} dV_g = \int \left(|W_g|^{2}+\frac{4}{n-2}|\ricc_g|^{2}-\frac{2}{(n-1)(n-2)}R_g^{2}\right) dV_g \,.
$$
 Such functionals have attracted a lot of attention from the mathematics' and physicists' communities in recent years. In particular, in \cite{CMM} (see also references therein) we proved rigidity results for critical metrics of the functional $\mathfrak{S}^2$ and for the functional $$\mathfrak{F}^{2}_t = \int |\ricc_g|^{2} dV_g +  t \int R^{2}_g  dV_g \,,$$ for suitable values of the parameter $t\in\mathbb{R}$. As far as the functional $\mathfrak{S}^2$ is concerned, we showed that in case $n=2$ all critical metrics are flat, and thus they are global minima of the functional. The same result holds also when $n=3$, under the additional hypothesis that $R_g\in L^q(M^3)$ for some $q\in(1,\infty)$. The dimension $n=4$ is special, as in this case critical metrics turn out to have harmonic scalar curvature. Thus, if $M^4$ is compact, then it is scalar flat, or it has constant scalar curvature and it is Einstein; on the other hand, if $M^4$ is complete, non-compact and $R_g\in L^q(M^4)$ for some $q\in(1,\infty)$, then a classical result of Yau implies that $M^4$ has constant scalar curvature, and hence it is scalar flat or Einstein with finite volume. Finally, when $n\geq5$ we showed that there exists $q^*>2$ such that a critical metric of $\mathfrak{S}^2$ having scalar curvature $R_g$ which is \emph{bounded from below} and satisfying $R_g\in L^q(M^n)$ for $q\in(1,q^*)$ must be scalar-flat, and thus it is a global minimum of the functional.

 We conjecture however that the condition that $R_g$ is bounded from below in the above results, when $n\geq5$, is indeed not necessary, i.e.
\begin{conge}
Let $(M^{n},g)$, $n\geq5$, be a complete critical metric of $\mathfrak{S}^{2}$ with finite energy. Then $(M^{n},g)$ is scalar flat, and thus a global minimum of the functional $\mathfrak{S}^2$.
\end{conge}
We also raise the question whether a finite energy assumption $R_g\in L^q(M^n)$ is necessary, in order to deduce that a complete non-compact critical metric of $\mathfrak{S}^{2}$ must be scalar flat, for $n\geq3$, $n\neq4$.

We recall that the Euler--Lagrange equation for a critical metric of $\mathfrak{S}^2$ can be computed by using variations with compact support and is given by
$$
2R \, \ricc - 2\nabla^{2} R + 2\Delta R\, g \,=\, \frac{1}{2} R^{2}\,g \,,
$$
or, equivalently,
\begin{equation}\label{eq}
R \,\ricc - \nabla^{2} R \,=\, \frac{3}{4(n-1)} R^{2}\,g \,,
\end{equation}
\begin{equation}\label{eq1}
\Delta R \,=\, \frac{n-4}{4(n-1)} R^{2} \,,
\end{equation}
where equation~\eqref{eq1} is just the trace of~\eqref{eq} (see also Proposition 4.66 in Besse's book \cite{Besse}; note that, in Corollary 4.67, Besse restricts the functional $\mathfrak{S}^2$ to unit-volume metrics).

The main result of this paper is the following  Theorem, where we give an affirmative answer to Conjecture 1.1 in  case    $n\geq10$, thus improving our \cite[Teorem 1.5]{CMM} in this range of dimensions: indeed, we have the following

\begin{teo}\label{teo}
Let $(M^{n},g)$, $n\geq 10$, be a complete critical metric of $\mathfrak{S}^{2}$ with finite energy, i.e. $R_g\in L^2(M^n)$. Then $(M^{n},g)$ is scalar flat, and thus a global minimum of the functional $\mathfrak{S}^2$.
\end{teo}


We actually show our result under the slightly weaker assumption that $R_g\in L^q(M)$ for some $q\in(1,q^*)$, for a suitable explicit $q^*>2$.

Our proof relies on a preliminary result that guarantees that a critical metric for $\mathfrak{S}^2$ with $R_g\in L^q(M)$ for some $q\in(1,\infty)$ and $n\geq5$ must have either identically vanishing or strictly negative scalar curvature. We explicitly note that there exist no critical metric for the $\mathfrak{S}^2$ functional with $R_g$ positive and $n\neq4$; indeed any such metric should have constant scalar curvature, thus it can exist only when $n=4$ and $(M,g)$ is Einstein, see \cite{Cat}. The weaker assumption $R_g\geq0$ is enough to conclude, using the strong maximum principle, in dimensions $n\leq4$. The case when the scalar curvature may change sign, and thus it must have infinite energy, remains an interesting and completely open problem.

In order to prove our main theorem we show that, under our assumptions, there exists no critical metric with negative scalar curvature. Indeed, if $R_g<0$ we can perform the conformal change of the metric
\begin{equation}\label{12}
\gt=|R_g|^\frac{6}{n-4}g
\end{equation}
 to produce a ``steady quasi-Einstein structure'' (in particular a steady Ricci soliton if $n=10$), i.e. it satisfies
$$
\ricc_{\gt}+\nabla^2_{\gt}f-\frac{n-10}{4(n-1)}df\otimes df=0
$$
with
$$
f=\frac{2(n-1)}{n-4}\log |R|.
$$
Note that $\frac{n-10}{4(n-1)}\geq0$ when $n\geq10$, while it is negative for $5\leq n\leq9$. This ``conformal technique'' has been used in the literature, for instance  by Anderson in the context of stationary space-times \cite{AndersonS}, by Fischer--Colbrie \cite{fisch} to study stable minimal surfaces in $\RR^3$ and, more recently, further exploited to study the stable Bernstein problem in $\RR^3$ by Catino--Mastrolia--Roncoroni \cite{CMR} or minimally immersed submanifold in the sphere by Magliaro--Mari--Roing--Savas-Halilaj \cite{MagliaroMariecc}.

Using lower bounds on the scalar curvature of $\gt$ we are able to deduce a gradient estimate on $R_g$, which then allows us to conclude that $R_g$ must actually vanish everywhere if it belongs to $L^q(M)$ for $q\in(1,q^*)$, for a suitable $q^*>2$, similarly as we did in \cite{CMM}. A key step in this construction is to show completeness of the conformal metric $\gt$, that we are able to obtain  when $n\geq10$.


We also explicitly comment on the estimates on the scalar curvature of $\gt$ that we use to obtain the gradient estimate on $R_g$:  we rely on results which are already available in the literature concerning nonnegativity of the scalar curvature for steady Ricci solitons (when $n=10$) and for steady quasi-Einstein manifolds (when $n\geq11$).


\

Note that Conjecture 1.1 for $5\leq n\leq9$ and the question whether a finite energy assumption on $R_g$ is necessary in order to prove that a critical metric for $\mathfrak{S}^2$ must be scalar flat remain still open.

\

The rest of the paper is organized as follows. In Section \ref{sec5} we show that the conformal change of the metric \eqref{12} gives rise to a quasi-Einstein manifold $(M^n,\gt)$, while in Section \ref{compl} we prove that $\gt$ is complete under the hypotheses of  Theorem \ref{teo}. In Section \ref{pr1}  we provide the proof of Theorem \ref{teo}

\

\begin{ackn}
%
The first and second authors are members of the {\em GNSAGA, Gruppo Nazionale per le Strutture Algebriche, Geometriche e le loro Applicazioni} of Indam. The third author is a member of {\em GNAMPA, Gruppo Nazionale per l'Analisi Matematica, la Probabilit\`{a} e le loro Applicazioni} of Indam. The first and second authors are partially funded by 2022 PRIN project 20225J97H5 ``Differential Geometric Aspects of Manifolds via Global Analysis''.
\end{ackn}

\

\section{Conformal quasi-Einstein manifolds}\label{sec5}

From now on we will drop the subscript $g$ in the notation of geometric objects. First of all we recall the following Lemma (see Lemma 5.1 in \cite{CMM}):

\begin{lemma}\label{l-hdnonpos}\ Let $(M^{n},g)$, $n\geq 5$, be a complete, non-compact, critical metric of $\mathfrak{S}^2$ with $R\in L^{q}(M^{n})$ for some $1<q<\infty$. Then $(M^{n},g)$ has non-positive scalar curvature.
\end{lemma}
%
%
Now, if $n\geq 5$, then by \eqref{eq1} $R$ is subharmonic, therefore Lemma \ref{l-hdnonpos} and the strong maximum principle imply the following
\begin{cor}\label{c-strong} Let $(M^{n},g)$, $n\geq5$, be a complete, non-compact critical metric of $\mathfrak{S}^2$ with $R\in L^{q}(M^{n})$ for some $1<q<\infty$. Then $(M^{n},g)$ is either scalar flat or it has negative scalar curvature.
\end{cor}

From now on we will assume that $(M^{n},g)$, $n\geq 5$, is a complete, non-compact, critical metric of $\mathfrak{S}^2$ with $R\in L^{q}(M^{n})$ for some $1<q<\infty$ and with negative scalar curvature. Let $u:=-R>0$ on $M$. From the critical equations, we have
\begin{equation}\label{equ}
\ricc =\frac{\nabla^{2} u}{u}- \frac{3}{4(n-1)} u\,g \,,
\end{equation}
\begin{equation}\label{eq1u}
\Delta u \,=- \frac{n-4}{4(n-1)} u^{2} \,,
\end{equation}

\begin{prop}\label{pro-qe} Let $(M^{n},g)$, $n\geq 5$, be a critical metric of $\mathfrak{S}^2$ with negative scalar curvature. Then, for all $\RR\ni k\neq 0, \frac{1}{n-2}$, the conformal metric $$\gt=|R|^{2k}g=u^{2k}g$$ satisfies
\begin{equation}\label{eq-qe0}
\mathrm{Ric}_{\gt}+\nabla^2_{\gt}f-\frac{1+2k-(n-2)k^2}{[(n-2)k-1]^2}df\otimes df=\frac{(n-4)k-3}{4(n-1)}e^{\frac{1-2k}{(n-2)k-1}f}\,\gt
\end{equation}
with
$$
f=[(n-2)k-1]\log |R|=[(n-2)k-1]\log u.
$$
\end{prop}
\begin{proof}
Since $f=[(n-2)k-1]\log |R|=[(n-2)k-1]\log u$, we  have
  \[
  df = [(n-2)k-1]\frac{dR}{R}=[(n-2)k-1]\frac{du}{u}
  \]
  and
   \[
   \nabla^2_gf= [(n-2)k-1]\left(\frac{\nabla^2_gu}{u}-\frac{du\otimes du}{u^2}\right),
   \]
   which implies
  \[
   \Delta_gf = [(n-2)k-1]\left(\frac{\Delta_gu}{u}-\frac{\abs{\nabla_gu}^2_g}{u^2}\right).
   \]
    On the other hand, from the standard formulas for a conformal change of the metric $\tilde{g}=e^{2\varphi}g$, $\varphi\in C^\infty(M)$, $\varphi>0$ we get
 \[
  \mathrm{Ric}_{\gt} = \mathrm{Ric}_g-(n-2)\left(\nabla_g^2\varphi-d\varphi\otimes d\varphi\right)-\left[\Delta_g \varphi+(n-2)\abs{\nabla_g \varphi}_g^2\right]g
  \]
  and
   \[
  \nabla^2_{\gt} f = \nabla^2_gf-\left(df\otimes d\varphi+d\varphi\otimes df\right)+g\left(\nabla f, \nabla \varphi\right)g.
  \]
  Note that, in our case, $\varphi=k\log u$; now we exploit the fact that $u$ satisfies equations \eqref{equ} and \eqref{eq1u} to conclude that
  \begin{align*}
 \mathrm{Ric}_{\gt}+\nabla^2_{\gt}f-\frac{1+2k-(n-2)k^2}{[(n-2)k-1]^2}df\otimes df=\frac{(n-4)k-3}{4(n-1)}e^{\frac{1-2k}{(n-2)k-1}f}\,\gt.
  \end{align*}

\end{proof}

\begin{cor} \label{cor-qe} Let $(M^{n},g)$, $n\geq 5$, be a critical metric of $\mathfrak{S}^2$ with negative scalar curvature. Then the conformal metric $$\gt=|R|^{\frac{6}{n-4}
}g$$ satisfies
\begin{equation}\label{eq-qe}
\mathrm{Ric}_{\gt}+\nabla^2_{\gt}f-\frac{n-10}{4(n-1)}df\otimes df=0
\end{equation}
with
$$
f=\frac{2(n-1)}{n-4}\log |R|.
$$
\end{cor}

\section{Completeness of the conformal metric}\label{compl}
In this Section we show that, under the hypotheses of  Theorem \ref{teo}, if $R$ is negative on $M$ then the conformal metric $$\gt=|R|^\frac{6}{n-4}g$$ is complete on $M$.
We have the following result, which holds for $n\geq10$:
\begin{prop}\label{pro-completeness} Let $(M^{n},g)$, $n\geq 10$, be a complete, non-compact, critical metric of $\mathfrak{S}^2$ with negative scalar curvature. Then, the conformal metric $$\gt=|R|^{\frac{6}{n-4}}g$$ is complete.
\end{prop}

\begin{proof}

Let $u=-R>0$ and $\frac{3}{n-4}\leq k<1$. As shown in \cite[Theorem 1]{fisch}, given a fixed reference point $o\in M^n$, we can construct a $\gt$-minimizing geodesic
$$
\gamma(s):[0,\infty)\to M^n,$$
where $s$ is the $g$-arclength. For the sake of completeness, we report the argument here. First of all, for every $\rho>0$, we consider the geodesic ball (of $g$) centered at $o$ of radius $\rho$, $B_\rho(o)$. Then, we first claim that there exists a $\gt$-minimizing geodesic joining $o$ to the closest (in $\gt$) boundary point of $B_\rho(o)$. Indeed, consider $u_\rho:=u+\eta$, where $\eta$ is a non-negative smooth function such that $\eta\equiv0$ in $B_\rho(o)$ and $\eta\equiv 1$ on $B_{\rho+1}^c(o)$. Since $u_\rho$ is bounded below away from $0$, the metric
$$\gt_\rho=u_\rho^{2k}g$$ is complete, and thus there exist $\gt_\rho$-minimizing geodesics joining $o$ to any boundary point of $B_\rho(o)$. Now let $\rho_i>0$ be a sequence of radii monotonically diverging to $+\infty$. For every $\rho_i>0$, since $\partial B_{\rho_i}(o)$ is compact, there exists $x_i\in \partial B_{\rho_i}(o)$ so that $x_i$ is closest (in $\gt_{\rho_i}$) to $o$. Let $\gamma_i$ be the $\gt_{\rho_i}$-minimizing geodesic joining $o$ to $x_i$. Note that $\gamma_i\subset \overline{B}_{\rho_i}(o)$, and since $u_{\rho_i}=u$ in $\overline{B}_{\rho_i}(o)$, then $\gamma_i$ is a $\gt$-minimizing geodesic.
We parametrize $\gamma_i$ with respect to $g$-arclength. In particular, since $|\dot \gamma_i(s)|_g=1$ for every $s$, up to subsequences, the sequence $ \dot\gamma_i(0)$ converges to a limit vector as $\rho_i\to \infty$. Thus, by ODE theory and Ascoli-Arzel\`a, $\gamma_i$ converge on compact sets of $[0,\infty)$ to a limiting curve $\gamma$ which is a $\gt$-minimizing geodesic and is parametrized by $g$-arclength.

We observe that the completeness of the metric $\gt=u^{2k}g$ will follow if we can show that the $\gt$-length of $\gamma$ is infinite, i.e.
$$
\int_\gamma\,d\tilde{s} = \int_\gamma u^{k}\,ds=+\infty.
$$
Indeed, by construction, the $\gt$-length of every other divergent geodesic starting from $o$ (i.e. its image does not lie in any ball $B_\rho(o)$) must be greater than or equal to that of $\gamma$.

Since $\gamma$ is $\gt$-minimizing, by the second variation formula one has
\begin{equation}\label{eq-sv}
\int_0^{\tilde{r}}(n-1)\left(\frac{d\varphi}{d\tilde{s}}\right)^2-\widetilde{R}_{11}\varphi^2\,d\tilde{s}\geq 0,
\end{equation}
for all $\varphi\in V_{\tilde{r}}$, where we set
$$V_{\tilde{r}}=\{\varphi\in C^0([0,\infty))\,|\,\varphi(s)=\varphi(0)=0\,\,\forall s\geq A,\,\varphi\in C^2([0,A])\,\text{ for some }0<A<\tilde{r} \},$$
$V:=V_{\infty}$ and where $\tilde{r}$ is the length of $\gamma$ in the metric $\gt$ and
$$
\widetilde{R}_{11} = \widetilde{\mathrm{Ric}}\left(\frac{d\gamma}{d\tilde{s}},\frac{d\gamma}{d\tilde{s}}\right).
$$
From \cite[Appendix]{nelli}, we have
$$
\widetilde{R}_{11}=u^{-2k}\left\{R_{11}-k(n-2)(\log u)_{ss}-k\frac{\Delta u}{u}+k\frac{|\nabla u|^2}{u^2}\right\}
$$
where $R_{11} = \ricc(e_1,e_1)=\ricc\left(\frac{d\gamma}{ds},\frac{d\gamma}{ds}\right)$. Using the critical equation \eqref{equ} and \cite[Appendix]{nelli} we obtain
\begin{align*}
R_{11} &= \frac{\nabla^2_{11}u}{u}-\frac{3}{4(n-1)}u\, g_{11} \\
&= \nabla^2_{11}\log u + |(\log u)_s|^2 -\frac{3}{4(n-1)}u \\
&= (\log u)_{ss}-k|(\nabla \log u)^{\perp}|^2+|(\log u)_s|^2 -\frac{3}{4(n-1)}u
\end{align*}
where $(\nabla \log u)^{\perp}$ is the component of $\nabla \log u$ perpendicular to $\frac{d\gamma}{ds}$. Therefore, from \eqref{eq1u}, we get
\begin{align*}
\widetilde{R}_{11}&=u^{-2k}\left\{[1-(n-2)k](\log u)_{ss}+\frac{(n-4)k-3}{4(n-1)}u+|(\log u)_s|^2 +k|\nabla \log u|^2-k|(\nabla \log u)^{\perp}|^2\right\}\\
&=u^{-2k}\left\{[1-(n-2)k](\log u)_{ss}+\frac{(n-4)k-3}{4(n-1)}u+(1+k)|(\log u)_s|^2 \right\}.
\end{align*}
From inequality \eqref{eq-sv}, since $k\geq \frac{3}{n-4}$, we obtain
\begin{align*}
(n-1)\int_0^{+\infty}& (\varphi_s)^2u^{-k}\,ds \\
&\geq \int_0^{+\infty} \varphi^2 u^{-k}\left\{[1-(n-2)k](\log u)_{ss}+\frac{(n-4)k-3}{4(n-1)}u+(1+k)|(\log u)_s|^2\right\}\,ds\\
&\geq \int_0^{+\infty} \varphi^2 u^{-k}\left\{[1-(n-2)k](\log u)_{ss}+(1+k)|(\log u)_s|^2\right\}\,ds,
\end{align*}
for all $\varphi\in V$. Integrating by parts, we obtain
$$
\int_0^{+\infty} \varphi^2 u^{-k}(\log u)_{ss}\,ds=-2\int_0^{+\infty} \varphi u^{-k-1}\varphi_s u_s\,ds+k\int_0^{+\infty} \varphi^2 u^{-k-2}(u_s)^2\,ds,
$$
and thus
\begin{align*}
(n-1)\int_0^{+\infty} (\varphi_s)^2u^{-k}\,ds &\geq -2[1-(n-2)k]\int_0^{+\infty}\varphi u^{-k-1}\varphi_s u_s\,ds \\
&\quad+ [1+2k-k^2(n-2)]\int_0^{+\infty} \varphi^2 u^{-k-2}(u_s)^2\,ds.
\end{align*}
Let now  $\varphi=u^k\psi$, with $\psi\in V$. We have
\begin{align*}
  \varphi^2u^{-k}&=u^{k}\psi^2,\\ \varphi_s&=k\psi u^{k-1}u_s+u^k\psi_s,\\ (\varphi_s)^2u^{-k}&=k^2\psi^2u^{k-2}(u_s)^2+u^k(\psi_s)^2+2k\psi\psi_su^{k-1}u_s,
\end{align*}
and substituting in the previous relation we get
\begin{align}\label{bo}
(n-1)\int_0^{+\infty} (\psi_s)^2u^{k}\,ds &\geq -2(1+k)\int_0^{+\infty}\psi u^{k-1}\psi_s u_s\,ds \\\nonumber
&\quad+ [1-k^2]\int_0^{+\infty} \psi^2 u^{k-2}(u_s)^2\,ds.
\end{align}
Integration by parts gives
$$
I:=\int_0^{+\infty}\psi u^{k-1}\psi_s u_s\,ds= -\frac{1}{k}\int_0^{+\infty} (\psi_s)^2u^{k}\,ds-\frac{1}{k}\int_0^{+\infty} \psi\,\psi_{ss} u^{k}\,ds
$$
Moreover, for every $t>1$ and, completing the square, for every $\eps>0$, we have
\begin{align}\label{eq12}\nonumber
2(1+k)I &=2(1+k)t I+2(1+k)(1-t)I\\\nonumber
&=-\frac{2t(1+k)}{k} \int_0^{+\infty} u^k(\psi_s)^2\,ds-\frac{2t(1+k)}{k} \int_0^{+\infty} \psi\psi_{ss} u^{k}\,ds\\\nonumber
&\quad+2(1+k)(1-t)\int_0^{+\infty}\psi\psi_s u^{k-1}u_s\,ds\\
&= -\frac{2t(1+k)}{k} \int_0^{+\infty} u^k(\psi_s)^2\,ds-\frac{2t(1+k)}{k} \int_0^{+\infty} \psi\psi_{ss} u^{k}\,ds\\\nonumber
&\quad+(1+k)(t-1)\eps \int_0^{+\infty}\psi^2u^{k-2}(u_s)^2\,ds+\frac{(1+k)(t-1)}{\eps}\int_0^{+\infty} u^k(\psi_s)^2\,ds\\\nonumber
&\quad+\frac{(1+k)(1-t)}{\eps}\int_0^{+\infty}u^k\left(\psi_s+\eps u^{-1}u_s \psi\right)^2\,ds.
\end{align}
Since $k<1$, choosing
$$
\eps:=\frac{1-k}{t-1}
$$
we obtain
\begin{align*}
2(1+k)I &= -\frac{2t(1+k)}{k} \int_0^{+\infty} \psi\psi_{ss} u^{k}\,ds+(1-k^2)\int_0^{+\infty}\psi^2u^{k-2}(u_s)^2\,ds\\
&\quad+\left[\frac{(1+k)(t-1)^2}{1-k}-\frac{2t(1+k)}{k}\right]\int_0^{+\infty} u^k(\psi_s)^2\,ds\\
&\quad-\frac{(1+k)(1-t)^2}{1-k}\int_0^{+\infty}u^k\left(\psi_s+\frac{1-k}{t-1} u^{-1}u_s \psi\right)^2\,ds.
\end{align*}
Therefore, from \eqref{bo}, we obtain
\begin{align}\label{e127}\nonumber
0 &\leq \left[\frac{(1+k)(t-1)^2}{1-k}-\frac{2t(1+k)}{k}+(n-1)\right]\int_0^{+\infty} u^k(\psi_s)^2\,ds-\frac{2t(1+k)}{k} \int_0^{+\infty} \psi\psi_{ss} u^{k}\,ds\\
&\quad -\frac{(1+k)(1-t)^2}{1-k}\int_0^{+\infty}u^k\left(\psi_s+\frac{1-k}{t-1} u^{-1}u_s \psi\right)^2\,ds
\end{align}
for every $t>1$. Let
$$
P(t):=\frac{(1+k)(t-1)^2}{1-k}-\frac{2t(1+k)}{k}+(n-1)
$$
A computation shows that $P(t)\leq0$ for some $t>1$ if and only if
$$
(1+k)(1-k)[1+2k-(n-2)k^2]\geq 0.
$$
Choose $k=\frac{3}{n-4}$.

\medskip

If $n>10$, then $(1+k)(1-k)[1+2k-(n-2)k^2]>0$ and thus $P(t)<0$ for some $t>1$. Therefore, we deduce
$$
0\leq -\int_0^{+\infty} u^{k}(\psi_s)^2\,ds-C \int_0^{+\infty} u^{k}\psi\psi_{ss} \,ds
$$
for some $C>0$ and every $\psi \in V$. Now we choose $\psi=s\eta$ with $\eta$ smooth with compact support in $[0,+\infty)$: thus
$$
\psi_s=\eta+s\eta_s,\quad\psi_{ss}=2\eta_s+s\eta_{ss},
$$
and we get
$$
\int_0^{+\infty} u^{k}\eta^2\,ds\leq \int_0^{+\infty} u^{k}\left(-2(C+1)s\eta\eta_s-Cs^2\eta\eta_{ss}-s^2(\eta_s)^2\right)\,ds.
$$
Choose $\eta$ so that $\eta\equiv 1$ on $[0,R]$, $\eta\equiv 0$ on $[2R,+\infty)$ and with $|\eta_s|$ and $|\eta_{ss}|$ bounded by $C/R$ and $C/R^2$, respectively, for $R\leq s\leq 2R$ and for some $C$ independent of $R$. Then
$$
\int_0^R u^{k}\,ds\leq \int_0^{+\infty} u^{k}\eta^2\,ds \leq C \int_R^{+\infty}u^{k}\,ds
$$
for some $C>0$ independent of $R$. We conclude that
$$
\int_0^{+\infty}u^{k}\,ds =+\infty,
$$
i.e. $\gt=u^{2k} g=u^{\frac{6}{n-4}}g$ is complete, if $n>10$.

\medskip

If $n=10$, then $k=1/2$ and $(1+k)(1-k)[1+2k-(n-2)k^2]=0$. In this case. it is easy to verify that $P(t)=3(t-2)^2$. Choose $t=2$. From \eqref{e127}, since $\eps=1/2$, we obtain
$$
\int_0^{+\infty}u^k\left(\psi_s+\frac12u^{-1}u_s \psi\right)^2\,ds\leq -C\int_0^{+\infty} u^{k}\psi\psi_{ss} \,ds
$$
for some $C>0$ and for every $\psi \in V$. Assume, by contradiction, that $u^k$ is integrable. Choosing again $\psi=s\eta$ with $\eta$ smooth so that $\eta\equiv 1$ on $[0,R]$, $\eta\equiv 0$ on $[2R,+\infty)$ and with $|\eta_s|$ and $|\eta_{ss}|$ bounded by $C/R$ and $C/R^2$, respectively, for $R\leq s\leq 2R$ and for some $C$ independent of $R$, we get that the right hand side tends to zero as $R$ tends to $+\infty$. By Fatou's lemma we obtain $su^{-1}u_s = -2$. Therefore, $u(s)=Cs^{-2}$, which contradicts the fact that $u^k=u^{1/2}$ is integrable. Therefore  $\gt=u^{2k} g=u g$ is complete also if $n=10$.
\end{proof}

\

\section{Proof of Theorem \ref{teo}}\label{pr1}

\begin{proof}[Proof of Theorem \ref{teo}] Let $(M^{n},g)$, $n\geq 10$, be a complete critical metric of $\mathfrak{S}^2$ with $R\in L^q(M^n)$ for some $1<q<q^*=\frac{7n-10}{2(n-4)}$. First of all, if $M^n$ is compact, then integrating \eqref{eq1} over $M^n$ we get $R\equiv 0$ on $M^n$. In case $M^n$ is non-compact, from Corollary \ref{c-strong}, either $R\equiv 0$ or $R<0$ on $M^n$. In the latter case we consider the conformal metric $$\gt=|R|^{\frac{6}{n-4}}g,$$ which is complete by Proposition \ref{pro-completeness} and satisfies \eqref{eq-qe}. In particular $(M^n,\gt)$ is a complete steady gradient Ricci soliton, if $n=10$, or a complete steady quasi-Einstein manifold, if $n>10$. In both cases, it is well known (see \cite{BLChen} and \cite[Theorem 1.4]{wang}) that the scalar curvature of $\gt$ must be nonnegative. By the formula for the conformal change, we obtain
\begin{align}\label{eq-rstima}\nonumber
0&\leq\widetilde{R}=e^{-2w}\left(R-2(n-1)\Delta w-(n-1)(n-2)|\nabla w|^2\right)\\\nonumber
&=u^{-\frac{6}{n-4}}\left(-u-\frac{6(n-1)}{n-4}\frac{\Delta u}{u}+\frac{6(n-1)}{n-4}\frac{|\nabla u|^2}{u^2}-\frac{9(n-1)(n-2)}{(n-4)^2}\frac{|\nabla u|^2}{u^2}\right)\\
&=u^{-\frac{6}{n-4}}\left(\frac12 u-\frac{3(n-1)(n+2)}{(n-4)^2}\frac{|\nabla u|^2}{u^2}\right)
\end{align}
where we used $w=\frac{3}{n-4}\log u$ as in the proof of Proposition \ref{pro-qe}, $R=-u$ and  \eqref{eq1u}. Thus
\begin{equation}\label{eq-gradest}
|\nabla u|^2 \leq \frac{(n-4)^2}{6(n-1)(n+2)}u^3.
\end{equation}
Fixing $O\in M^n$, arguing as in \cite[Corollary 5.7]{CMM}, from \eqref{eq-gradest}, we obtain
\begin{equation}\label{20}
u(x)\geq \frac{c_1}{c_2+d_{g}(x,O)^2}
\end{equation}
for every $x\in M^n$ and some positive constants $c_i=c_i(n,u(O))$, i=1,2. Now the result follows as in the proof of \cite[Theorem 1.5]{CMM}. For the sake of completeness we include the proof.

\smallskip

Let $\eta$ be a smooth cutoff function such that $\eta\equiv1$ on $B_s(O)$, $\eta\equiv0$ on $B_{2s}^c(O)$, $0\leq\eta\leq1$ on $M^n$ and $|\nabla\eta|\leq\frac{c}{s}$ for every $s\gg1$ with $c>0$ independent of $s$.

Then, using \eqref{eq1u} and \eqref{eq-gradest} we get
\begin{align*}
\frac{n-4}{4(n-1)}\int_M u^q \eta^2\,dV_{g} &= -\int_M \Delta u\, u^{q-2}\eta^2\,dV_{g} \\
&= (q-2) \int_M |\nabla u|^2 u^{q-3}\eta^2\,dV_{g} +2\int_M u^{q-2}\langle \nabla u,\nabla \eta\rangle \eta\,dV_{g}\\
&\leq \frac{(n-4)^2 \,\max\set{q-2, 0}}{6(n-1)(n+2)}\int_M u^{q}\eta^2\,dV_{g} +\frac{C}{s}\int_{B_{2s}(O)\setminus B_s(O)}u^{q-\frac{1}{2}}\,dV_{g},
\end{align*}
for some $C>0$. By \eqref{20}
\begin{equation}\label{11}
\frac{n-4}{4(n-1)}\int_M u^q \eta^2\,dV_{g}\leq\frac{(n-4)^2\,\max\set{q-2, 0}}{6(n-1)(n+2)} \int_M u^{q}\eta^2\,dV_{g} +C\frac{(1+s^2)^\frac{1}{2}}{s}\int_{B_s^c(O)}u^{q}\,dV_{g}.
\end{equation}
Thus, if $u\in L^q(M^n)$, we obtain
\begin{equation*}
\frac{(n-4)^2}{6(n-1)(n+2)}\left[\frac{3(n+2)}{2(n-4)}-\max\set{q-2, 0}\right]\int_M u^q \eta^2\,dV_{g}\leq C\frac{(1+s^2)^\frac{1}{2}}{s}\int_{B_s^c(O)}u^{q}\,dV_{g}\longrightarrow 0,
\end{equation*}
as $s\to +\infty$.
This yields $u\equiv0$, if 
$$
1<q<q^*=2+\frac{3(n+2)}{2(n-4)}=\tfrac{7n-10}{2(n-4)},
$$ which is a contradiction. This concludes the proof of Theorem \ref{teo}.
\end{proof}

\begin{rem} Note that the gradient estimate \eqref{eq-gradest} improves the one in \cite[Lemma 5.5]{CMM} (see also Remark 5.6 there for the explicit expression of the constant), since it is possible to show that, for every $n\geq 10$, the constant $\tfrac{(n-4)^2}{6(n-1)(n+2)}$ is always smaller than the corresponding constant appearing there. As a consequence, we see that the conclusion of Theorem \ref{teo} follows assuming $R\in L^q(M^n)$ with $1<q<\tfrac{7n-10}{2(n-4)}$, thus improving, for $n\geq 10$, \cite[Theorem 1.5]{CMM} also in this respect.

\end{rem}

\

\noindent{\bf Data availability statement}

\noindent Data sharing not applicable to this article as no datasets were generated or analysed during the current study.

\
\

\noindent{\bf Conflict of interest statement}

\noindent On behalf of all authors, the corresponding author states that there is no conflict of interest.


\

\

\parindent=0pt

\end{document}